\documentclass[11pt]{article}
\usepackage[utf8]{inputenc}
\usepackage{amsmath,amssymb,amsfonts,amsthm}
\usepackage{fullpage}
\usepackage{enumerate}
\usepackage{graphicx}

\newtheorem{theorem}{Theorem}
\newtheorem{lemma}[theorem]{Lemma}

\newtheorem{claim}{Claim}

\newtheorem{conjecture}[theorem]{Conjecture}

\newcommand\C{\mathcal C}

\title{Strong log-convexity of genus sequences}

\author{
Bojan Mohar\thanks{Supported in part by an NSERC Discovery Grant R832714 (Canada), and in part by the ERC Synergy grant KARST (European Union, ERC, KARST, project number 101071836).}~\thanks{On leave from:
FMF, Department of Mathematics, University of Ljubljana, Ljubljana, Slovenia.}\\
Department of Mathematics, Simon Fraser University, Burnaby, BC, Canada}

\date{}

\begin{document}

\maketitle

\begin{abstract}
  For a graph $G$, and a nonnegative integer $g$, let $a_g(G)$ be the number of $2$-cell embeddings of $G$ in an orientable surface of genus $g$ (counted up to the combinatorial homeomorphism equivalence).
  In 1989, Gross, Robbins, and Tucker \cite{GrRoTu89} proposed a conjecture that the sequence $a_0(G),a_1(G),a_2(G),\dots$ is log-concave for every graph $G$. This conjecture is reminiscent to the Heron–Rota–Welsh Log Concavity Conjecture that was recently resolved in the affirmative by June Huh et al.~\cite{Huh12,AHK18}, except that it is closer to the notion of $\Delta$-matroids than to the usual matroids. In this short paper, we disprove the Log Concavity Conjecture of Gross, Robbins, and Tucker by providing examples that show strong deviation from log-concavity at multiple terms of their genus sequences.
\end{abstract}



\section{Introduction}

The \emph{genus sequence} (also called the \emph{genus distribution}) of a graph $G$ is the sequence $(a_g(G))_{g\ge0}$, whose term $a_g(G)$ gives the number of 2-cell embeddings of $G$ in the orientable surface of genus $g$ (up to ambient homeomorphisms). Due to a variety of applications within mathematics, theoretical physics and theoretical computer science (see, e.g., \cite{LaZv04}), there is abundant literature on the genus distribution of graphs. The earlier papers were mainly concerned with the \emph{average genus} computations \cite{St83,Ar88}. One of the initial motivations for considering the variety of all 2-cell embeddings of a graph was the possibility of performing local optimization for the minimum genus problem, with the hope to obtain good approximations (see, e.g., \cite{GrFu87} or \cite{Wh94,ChGrRi95}).
Later it became evident that genus distributions are intimately related to the enumeration of rooted maps on surfaces \cite{MaLiuWe06}. We refer to \cite{LaZv04} for extensive treatment of map enumeration and its applications in different areas of mathematics,
such as theoretical physics (quantum field theory, string theory, Feynman diagrams, the KdV equation), matrix integrals, algebraic number theory (algebraic curves, Galois theory via Grothendiek's \emph{dessins d’enfants}, geometry of moduli spaces of complex algebraic curves), singularity theory.
Even the very special case of one-vertex graphs (bouquet of loops) bears interesting applications, in particular the relationship to the Hopf algebra of chord diagrams and algebraic invariants of knots and links \cite{LaZv04}. The one-vertex maps are isomorphic, via surface duality, to unicellular maps that have received a lot of attention, see, e.g. \cite{BeCh11,Cha10,Cha11}.

Embedding distributions for specific families have been the subject of extensive research. Families containing vertices of arbitrarily large degrees include bouquets of loops \cite{GrRoTu89,ChGrRi94,KwSh02}, dipoles \cite{KwLee93,ViWi07}, bouquet of dipoles \cite{KiLe98}, and others \cite{ShLiu10}. More often, families consist of graphs whose structure resembles a long path (see, e.g., \cite{FuGrSt89,Sta91,Te97,Te00,LiHaZh10,MaLiuWe06,HaHeLiuWe07,WaLiu08})
or a long cycle (\cite{ShLiu08,ZeLiu11}); in such structures, the vertices of a path or a cycle are replaced by copies of a small graph.

While considering embedding distributions of various infinite families, Gross, Robbins, and Tucker \cite{GrRoTu89} proposed what became known as the {\bf Log-Concavity Genus Distribution Conjecture} (shortly {\bf LCGD Conjecture}):

\begin{conjecture}[Gross, Robbins, and Tucker, 1989]
\label{conj:LCGD}
  The genus distribution of every graph is log-concave.
\end{conjecture}

The conjecture has been confirmed for many graph families, including bouquets of loops \cite{GrRoTu89}, dipoles \cite{KwLee93,ViWi07}, graphs constructed by vertex and edge-amalgamations \cite{GMTW15}, linear sequences of graphs \cite{CGM15a,CGMT20a,CGMT20b}, ring-like families \cite{GMT14b,CGM15b}, Halin graphs \cite{GMT14b}, and others.

The LCGD Conjecture is reminiscent on another famous problem, the Heron–Rota–Welsh log-concavity conjecture \cite{Heron72,Rota71,Welsh71}, which asserts that the coefficients of the characteristic polynomial of a matroid form a log-concave sequence. A weaker conjecture that the coefficients are unimodal has been proposed by Read. In 2012, Huh used algebraic geometry to prove Read’s unimodality conjecture \cite{Huh12} for graphs, and the more general Heron–Rota–Welsh conjecture for matroids represented over a field of characteristic 0. The case of matroids representable over a field of a non-zero characteristic was resolved a bit later by Huh and Katz \cite{HuKa12}. Finally, in 2018 the Heron–Rota–Welsh conjecture was proved in full generality by Adiprasito, Huh, and Katz \cite{AHK18}. All these exciting developments gave hope to resolve the LCGD Conjecture since the variety of 2-cell embeddings is related to the notion of \emph{$\Delta$-matroids} \cite{Bo87,Bo89,Mo08,CMNR19} and to the Bollob\'as-Riordan polynomial \cite{BoRi01}, both of which deal with graphs on surfaces and involve contraction-deletion recurrences.

The main results of this paper are Theorems \ref{thm:main} and \ref{thm:second main result} that refute the LCGD Conjecture in a rather spectacular (and unexpected) way. Let us state it first in a weaker form.

\begin{theorem}
\label{thm:main weaker}
  For every $k\ge1$ there is a $4$-connected graph whose genus sequence has at least $k$ terms at which the log-concavity is violated.
\end{theorem}

We use standard graph theory language \cite{Di18} and follow the usual terminology of topological graph theory \cite{MT01}.
Throughout the paper, $G$ will always be a connected simple graph, and by a surface we will refer to a closed orientable surface. Let us observe that the analogue of our main result holds for nonorientable surfaces as well.

\section{Multiple log-convexity of genus sequences is viable}

Let $(a_n)_{n\in \mathbb N}$ be a sequence of non-negative numbers. The $k$th term $a_k$ ($k\ge1$) of this sequence is said to be \emph{log-concave} if $a_k^2 \ge a_{k-1} a_{k+1}$, and it is \emph{log-convex} if $a_k^2 \le a_{k-1} a_{k+1}$. If the corresponding inequality is strict, we say that the term is \emph{strictly log-concave} or \emph{strictly log-convex}, respectively. The sequence is \emph{log-concave} if all of its terms are log-concave. These definitions extend to any finite sequence, which can be viewed as an infinite sequence with zeros appended at the end to make it infinite.

The afore-mentioned LCGD Conjecture \ref{conj:LCGD} of Gross, Robbins, and Tucker \cite{GrRoTu89} states that the genus sequence of every graph is log-concave. In this section we refute this long-standing open problem by proving that there are graphs whose genus sequence has many strictly log-convex terms.

\begin{theorem}
\label{thm:main}
  For every integer $g\ge0$ and every $k\ge2$ there is a $4$-connected graph $G_{g,k}$ of genus $g$ such that the terms $a_{g+1},a_{g+3},a_{g+5},\dots,a_{g+2k-1}$ of its genus sequence are all strictly log-convex.
\end{theorem}

In the proof of this result we will need several properties about combinatorics of 2-cell embeddings.\footnote{An embedding of a graph in a surface $S$ is a \emph{$2$-cell embedding} if each face of the embedding is homeomorphic to an open disk.}
Suppose that a graph $G$ is 2-cell embedded in a closed (orientable) surface $S$ of genus $g$.
A closed walk or a cycle in $G$ that bounds a face in this embedding is said to be a \emph{facial walk} or \emph{facial cycle}. A cycle $C$ of $G$ is \emph{contractible} if it is contractible viewed as a curve on the surface. This is the same as saying that $C$ splits the surface into two components, one of which is a disk. More generally, $C$ is \emph{surface separating} if it splits the surface into two components.

Combinatorially, $C$ is a \emph{non-separating cycle} of the graph $G$ if the subgraph $G-V(C)$ is connected; the cycle $C$ is an \emph{induced cycle} of $G$ if it is induced as the subgraph of $G$, i.e. it has no chords in $G$. Realizing that facial cycles of 3-connected plane graphs (i.e. graphs of 3-dimensional convex polyhedra) are induced and non-separating, Tutte \cite{Tu63} called such cycles \emph{peripheral}. Due to this relationship, we say that a 2-cell embedding of a graph in a surface is a \emph{polyhedral embedding} if every face is bounded by a peripheral cycle.

If a peripheral cycle in a 2-cell embedded graph is surface separating, then it must be facial, see \cite{MT01} for more details. For further use we state this simple fact as a lemma.

\begin{lemma}\label{lem:peripheral cycle}
  Suppose that $C$ is a peripheral cycle of a graph $G$ that is $2$-cell embedded in a surface $S$. If $C$ is not facial, then it is surface non-separating.
\end{lemma}

More generally, we say that a family $\C$ of peripheral cycles in $G$ is \emph{peripheral} if any two cycles in the family are disjoint and nonadjacent (i.e., no two vertices on distinct cycles are adjacent) and $G-\cup\{V(C)\mid C\in\C\}$ is connected.

The following result is a simple generalization of Lemma \ref{lem:peripheral cycle}.

\begin{lemma}\label{lem:peripheral family}
  Suppose that $\C$ is a peripheral family of cycles of a graph $G$ that is $2$-cell embedded in a surface $S$ of genus $g$. If none of the cycles in $\C$ is facial, then $|\C|\le g$.
\end{lemma}

\begin{proof}
  Each cycle in $\C$ is surface nonseparating and has a neighbor on each side (otherwise it would be facial). Those neighbors are not on any of the other cycles in the family. This implies that the union of all cycles in $\C$ is surface non-separating. Consequently, repeatedly cutting the surface along any cycle in $\C$ decreases the genus by $1$ (easily seen by using Euler characteristic computation), and hence the number of cycles is at most $g$.
\end{proof}

Two facial walks $W_1,W_2$ are said to be \emph{cofacial} if there is a facial walk $W$ that contains a vertex in $W_1$ and a vertex in $W_2$. In particular, two facial walks that intersect are cofacial. The facial walk $W$ is said to be a \emph{cofacial connector} for $W_1$ (and at the same time also for $W_2$). Given a set $\C$ of facial cycles, and $C\in \C$, we denote by $\Delta(C)$ the number of different cofacial connectors for $C$. The set $\C$ is \emph{sparse} if $\Delta(C)\le1$ for each $C\in\C$.

\begin{lemma}\label{lem:non cofacial faces}
   Suppose that $G$ is a 3-connected graph with a polyhedral embedding in some surface and $\C$ is a sparse collection of its facial cycles. Suppose that $G$ has another embedding $\Pi$ in which none of the cycles in $\C$ is facial. Then the genus of the embedding $\Pi$ is at least $|\C|$.
\end{lemma}

\begin{proof}
  Let us first observe that the cycles in $\C$ are peripheral since they are faces of a polyhedral embedding. Moreover, sparsity implies that cycles in $\C$ are pairwise disjoint and pairwise nonadjacent.
  The basic property used in the proof will be the following. For $C\in \C$, let $N(C)$ be the set of vertices that are not on $C$ but have a neighbor on $C$, and let $W(C)$ be the set of vertices that are cofacial with a vertex on $C$ but do not belong to $\cup\C$. The vertices in $N(C)$ are ordered clockwise as $u_1,u_2,\dots,u_p$, and each $u_i$ is connected with $u_{i+1}$ via part $W_i$ of a facial walk that does not intersect $C$. Since $\C$ is sparse, we may assume that the facial walks $W_i$ (except possibly for $W_p$) are disjoint from $\cup\C$. This shows that for every $v_i,v_j\in N(C)$ there is a path in $G-V(\cup\C)$ connecting them.

  Each cycle in $\C$ is surface nonseparating in $\Pi$. Moreover, the family $\C$ is peripheral. To see this, it is clear that $\C$ is induced. To show that it is nonseparating, take vertices $x,y\in V(G)\setminus V(\cup\C)$ and consider a path from $x$ to $y$ in $G$. If the path contains a segment $Q$ whose first and last vertex are out of $\C$ but other vertices on $Q$ are in one of the cycles $C\in \C$, then we can replace $Q$ by a path in $G-V(\cup\C)$ connecting them (as proved above). By repeating this for all such subpaths $Q$ we obtain an $(x,y)$-walk in $G-V(\cup\C)$, showing that $x$ and $y$ are connected in $G-V(\cup\C)$. This shows that $\C$ is a peripheral family.

  The proof is now easily concluded by applying Lemma \ref{lem:peripheral family}.
\end{proof}

Our last ingredient is a bound on the number of disjoint nonhomotopic cycles on a surface. By a \emph{homotopy} between two cycles of the graph embedded in a surface we mean the homotopy between some of their orientations viewed as loops on the surface.

\begin{lemma}\label{lem:disjoint nonhomotopic}
  Suppose that $\C$ is a family of pairwise disjoint and pairwise nonhomotopic and noncontractible cycles on a surface of genus $g\ge1$. Then $|\C|\le 3g-2$.
\end{lemma}

The proof of the lemma can be found in \cite{MaMo92} (see also \cite[Proposition 4.2.6]{MT01}). The bound of $3g-2$ can actually be improved to $3g-3$ if $g\ge2$ (which is best possible).

We are ready to proceed with our main construction.

  For $m\ge3$, let $\widehat{C}_m$ be the 4-regular \emph{antiprism} graph on $2m$ vertices that is obtained from the prism $C_m\Box K_2$ by adding diagonals into quadrangular faces. More precisely, $\widehat{C}_m$ is obtained from two copies of the cycle of length $m$ with vertices $v_1^0,\dots, v_m^0$ and $v_1^1,\dots, v_m^1$ (respectively) and adding all edges $v_i^0v_i^1$ and $v_i^0v_{i+1}^1$ (index $i+1$ considered modulo $m$), $i=1,\dots,m$.

\begin{figure}
\begin{center}
\includegraphics[width=125mm]{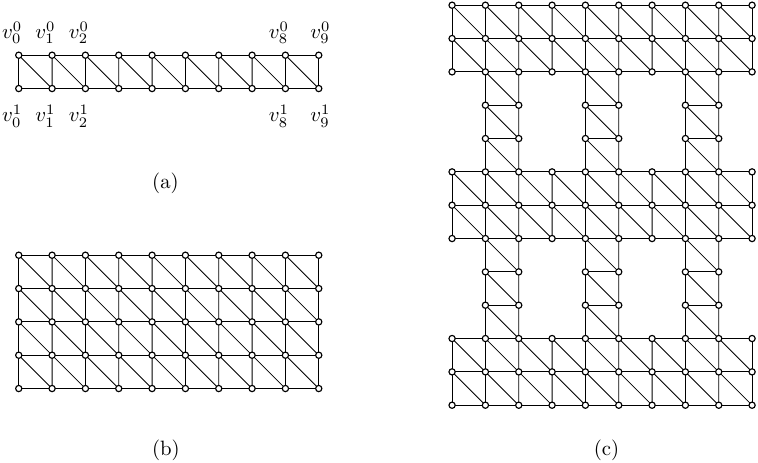}
\end{center}
\caption{(a) The antiprismatic path of length 9 gives the antiprism $\widehat{C}_9$ after identifying the left edge $v_0^0 v_0^1$ with the one on the right. (b) The stacked antiprism $\widehat{C}_9^4$ of width 9 and height 4. (c) The genus 0 cylinder $H_m^{k,b,d}$ with parameters $m=9$, $k=2$, $b=2$, and $d=3$. The vertices and edges on the left are identified with the vertices and edges on the right side to give a planar graph on the cylinder.}
\label{fig:1}
\end{figure}

  Next, we define the \emph{stacked antiprism} $\widehat{C}_m^b$ of \emph{width} $m$ and \emph{height} $b$ by stacking $b$ copies of $\widehat{C}_m$ on the top of each other (identifying the cycle $v_1^0v_2^0\dots v_m^0$ of the bottom one with the cycle $v_1^1v_2^1\dots v_m^1$ of the copy at the top of it). See Figure \ref{fig:1}(b) for an example.

\begin{figure}
\begin{center}
\includegraphics[width=15cm]{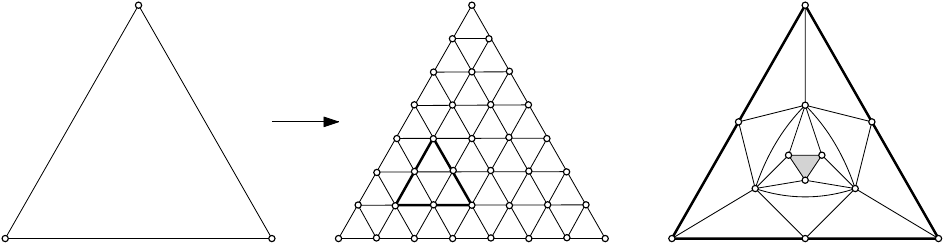}
\end{center}
\caption{(i) Replacing a triangular face by the triangular grid of side length $l$, illustrated here with $l=7$. (ii) Replacing four triangles inside the triangular grid so that a face with all vertices of degree 4 is obtained.}
\label{fig:triangular subdivision}
\end{figure}
  
A graph that is 2-cell embedded in a surface $S$ is a \emph{near-triangulation} if all its faces with the exception of a single face $F$ are triangles. The exceptional face $F$ is called the \emph{outer face}. It can be of any length, but each edge on $F$ is contained in one of the triangular faces. If $k\ge1$, we define similarly a \emph{$k$-near-triangulation} of a surface if there are precisely $k$ faces that are not required to be of length 3.

\begin{lemma}\label{lem:desired near-triangulation}
Let $g\ge0$ and $a,b\ge1$ be integers. There exists a near-triangulation $T_g$ with the following properties:
  \begin{itemize}
   \item[(1)] The embedding of $T_g$ is in the orientable surface of genus $g$.
   \item[(2)] The graph $T_g$ is 4-connected.
   \item[(3)] The maximum vertex degree in $T_g$ is at most 7 and any two vertices of degree $7$ are at least distance $a$ apart. The outer face has length $3a$ and each of its vertices has degree $4$. 
   \item[(4)] $T_g$ contains a peripheral family of noncontractible cycles $D_1,\dots,D_g$ that are pairwise at distance at least $2b$, all are of length $3a$, all of their vertices have degree $6$, and each cycle $D_i$ passes transversally through each of its vertices, i.e., the two edges of $D_i$ incident with a vertex $v$ are opposite in the local rotation around $v$. Moreover, the vertices at distance at most $b$ from $D_i$ form a triangulation of the cylinder isomorphic to the stacked antiprism $\widehat{C}_{3a}^{2b}$. 
   \item[(5)] By cutting the surface along the cycles $D_1,\dots,D_g$ we obtain a surface of genus $0$ (with $2g$ boundary components corresponding to the cut). 
  \end{itemize}
\end{lemma}

\begin{proof}
  We start with a $5$-connected triangulation of the plane, whose maximum degree is 6 and has at least $2g+1$ faces. Now, we replace each triangle with the triangular grid of side length $5$ (see Figure \ref{fig:triangular subdivision}, where the triangular grid of side length $7$ is shown). In $2g+1$ of these triangular grids, we replace the four triangles in the middle with the graph shown in Figure \ref{fig:triangular subdivision}. By doing this, we obtain a triangulation of the plane of maximum degree 7 that contains $2g+1$ triangles $F_0,F_1,\dots,F_{2g}$, whose vertices are all of degree 4. Now, we once again replace every triangle except $F_0,\dots,F_{2g}$ by the triangular grid of side length $a$. This gives rise to a $(2g+1)$-near-triangulation, whose outer faces $F_0',\dots,F_{2g}'$ correspond to $F_0,\dots,F_{2g}$, they are all of length $3a$ and all their vertices are of degree 4.
  
  Finally, for each $i\in \{1,\dots,g\}$, we take a copy of the stacked antiprism $\widehat{C}_{3a}^{2b}$ of width $3a$ and height $2b$ and identify its top and bottom face (each of length $3a$) with $F_{2i-1}'$ and $F_{2i}'$, respectively. This operation introduces $g$ handles and produces a near-triangulation of the orientable surface of genus $g$ (with outer face $F_0'$). The cycle of length $3a$ in the middle of the stacked antiprism is then denoted by $D_i$. 
  
  It is now easy to see that the obtained near-triangulation satisfies all stated conditions.
\end{proof}

With this lemma, we are ready to describe the construction for Theorem \ref{thm:main}.

\begin{proof}[Proof of Theorem \ref{thm:main}.]
  In order to construct the desired embedded graph $G_{g,k}$, we first take the disjoint union of $k+1$ stacked antiprisms $\widehat{C}_m^b$, where $b$ is a large positive integer (to be specified later), and where $m\ge6$ is a multiple of 3. For $j=1,\dots,k$, we join the $j$th stacked antiprism with three \emph{antiprismatic paths} of length $d$ (the value to be specified later) to the $(j-1)$st stacked antiprism. These antiprismatic paths will be referred to as the \emph{connectors}. The connectors start and end on the stacked antiprisms equally spaced (at distance $\tfrac{1}{3}m-1$ from each other). See Figure \ref{fig:1}(c) for an example. The resulting graph $H_m^{k,b,d}$ has a natural embedding into the cylinder, and is therefore planar. It is easy to see that $H_m^{k,b,d}$ is 4-connected, so its embedding in the plane is polyhedral. There are two faces $F_1,F_2$ of length $m$, one at the top and one at the bottom, and for each $j\in \{1,\dots,k\}$, there are $3$ faces of length $\tfrac{2}{3}m+2d-2$. All other faces are of length~$3$.

  Having described the graph $H_m^{k,b,d}$ embedded in the plane, we next fill each of its faces $F_1$ and $F_2$ of length $m$ with a near-triangulation, whose outer facial cycle is identified with the facial cycle of $F_1$ or $F_2$, such that the following properties hold:
  \begin{itemize}
   \item[(a)] Apart from the $3k$ long faces within $H_m^{k,b,d}$, all other faces in $G_{g,k}$ are triangles;
   \item[(b)] the embedding of $G_{g,k}$ is of genus $g$;
   \item[(c)] the graph $G_{g,k}$ is 4-connected;
   \item[(d)] the maximum vertex degree in $G_{g,k}$ is at most 7 and any two vertices of degree 7 are at distance at least 6 from each other;
   \item[(e)] every noncontractible cycle in $G_{g,k}$ has length at least $f(g,k):=6(g+2k+1)$;
   \item[(f)] $G_{g,k}$ contains a peripheral family of $g$ cycles whose pairwise distance and their distance from the subgraph $H_k$ is at least $f(g,k)$.
  \end{itemize}
  The two near-triangulations to fill-in the two faces $F_1$ and $F_2$ are easy to obtain by using Lemma \ref{lem:desired near-triangulation} with appropriate parameters. We can take the one in $F_1$ to have genus $g$ and the one in $F_2$ having genus 0. 
  
  Note that for every $g\ge0$ and $k\ge1$, the resulting embedded graph $G_{g,k}$ is 4-connected and is polyhedrally embedded in the surface of genus $g$. Herewith we will also assume that $m\ge f(g,k)$ and that $b=f(g,k)$.
  Let $\nu(g,k)$ be the number of vertices of $G_{g,k}$ not counting the vertices on the $3k$ connectors in $H_m^{k,b,d}$, i.e.
  $$\nu(g,k) = |V(G_{g,k})| - 6k(d+1).$$
  Now, we take the last parameter $d$ such that the following inequality holds:
  \begin{equation}\label{eq:d}
    d > 36k^2(\log_2 d + 7) + (g+2k)(217+\log_2 \nu(g,k)).
  \end{equation}
  This completes the description of the graphs $G_{g,k}$.
  The constructed polyhedral embedding of $G_{g,k}$ has genus $g$ and it will be denoted by $\Pi_0$.

  Next, we will obtain upper and lower bounds for the embedding coefficients $a_h(G_{g,k})$ for $g\le h\le g+2k+1$. Let us first observe that the constructed embedding of $G_{g,k}$ of genus $g$ is polyhedral and its face-width (see \cite{MT01} for the definition, which we omit here since it is not really essential) is at least $f(g,k) > 2g$. By a result of Robertson and Vitray \cite{RoVi90} (see also \cite{Mo92}), the embedding of $G_{g,k}$ is a minimum genus embedding, and the graph has no other embeddings of genus (at most) $g$ save the one with the same facial cycles as $\Pi_0$ but with reversed orientation. Thus,
  $$
     a_g(G_{g,k})=2.
  $$
  Let us now consider 2-cell embeddings of $G_{g,k}$ of genus $h=g+r$, where $1 \le r \le 2k$. 
  We will first give a lower bound and then an upper bound on $a_h(G_{g,k})$. The lower bound in Claim \ref{cl:lower bound} is of the form $2^{2d(r+q)}$, where $q:=\lfloor r/2\rfloor$. The upper bound is of the form $2^{2d(r+q)}p(n,d)$, where $p$ is a polynomial whose degree is bounded in terms of constants $g$ and $k$ and $n$ is the order of the graph. Observe that our construction gives that $n=O(g^2 k^2 d)$. For even $r$, the value of $q$ is the same as for $r+1$. We use the lower bound for $r-1$ and $r+1$ when $r$ is even, in which case $\lfloor (r-1)/2\rfloor \cdot \lfloor (r+1)/2\rfloor > 2\lfloor r/2\rfloor$, and this yields log-convexity when $d$ is large enough.

  Let us first give a lower bound on $a_h(G_{g,k})$.

  \begin{claim}\label{cl:lower bound}
     Suppose that $1\le r\le 2k$. Let $q:=\lfloor r/2\rfloor$. Then we have:
     $$
        a_{g+r}(G_{g,k}) \ge 2^{2d(r+q)}\binom{k}{q}.
     $$
  \end{claim}

\begin{figure}
\begin{center}
\includegraphics[width=110mm]{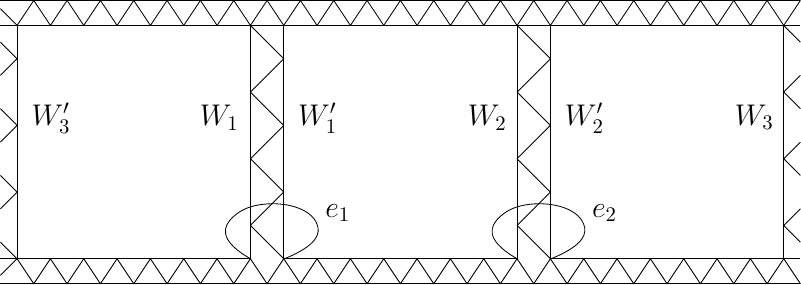}
\end{center}
\caption{Merging three big faces into a single face $F$.}
\label{fig:2}
\end{figure}

  \begin{proof}
    Let us consider the embedding of $G_{g,k}$ of genus $g+r$ that is obtained from the embedding $\Pi_0$ as follows. First, choose $q$ indices $j$ between $1$ and $k$ and for each such $j$ consider the two edges of the connectors on the $j$th stacked antiprism $\widehat C_m^b$. Now change the local rotation at the endpoints of these two edges so that these edges would emanate into the incident big faces. This change merges four triangular faces of $\Pi_0$ containing these two edges into two faces of length 4 and merges the three big faces into a single face. Hence the effect of this change is that the genus increases by two. Let $W_i,W_i'$ be the two paths on the $i$th connector ($i=1,2,3$), oriented from ``bottom'' to ``top'' (see Fig.~\ref{fig:2}). Then the big new face $F$ contains these paths in the following order: $F= W_1^- e_1 W_1'\cdots W_2^- e_2 W_2' \cdots W_3^- \cdots W_3' \cdots$, where $W_i^-$ denotes the traversal of the path in the opposite direction and $e_i$ is the flipped edge, i.e., the local rotation at its end is changed as indicated in Fig.~\ref{fig:2} so that the long faces merge as indicated above. Now, each of the $2d$ edges connecting $W_i$ and $W_i'$ can stay where it is or it can be flipped and reembedded into $F$, giving $2^{2d}$ different choices. By doing this for $i=1,2,3$ and for each of the $q$ chosen connectors, we get $2^{6dq}$ different embeddings, all of genus $g+2q$. This proves the claim when $r$ is even.

    If $r$ is odd, we select one of the remaining connecting triples and in that case flip the first edge of just one of the connectors. This will further increase the genus by one. In the same way as above we see that we can independently reembed each of the $2d$ edges between $W_1$ and $W_1'$ into the new big face $F$ which has $W_1$ and $W_1'$ on its boundary traversed in the opposite directions. This gives additional $2^{2d}$ possible choices to each of the previously described embeddings of genus $g+r-1$.
  \end{proof}

  The next claim gives an upper bound on $a_h(G_{g,k})$, where we will not try to optimize the bound.

  \begin{claim}\label{cl:upper bound}
     Suppose that $1\le r\le 2k$. Let $q:=\lfloor r/2\rfloor$ and $\nu=\nu(g,k)$. Then we have:
     $$
        a_{g+r}(G_{g,k}) \le \left((5!)^{31} \nu\right)^{g+2k} \cdot \left( 108d \right)^{36k^2} \cdot 2^{2d(q+r)} .
     $$
  \end{claim}

  \begin{proof}
  We split the vertices of $G=G_{g,k}$ into four sets, $V_A,V_B,V_C,V_D$, which are defined as follows. First, $V_D$ contains for each of the $3k$ connectors the six vertices in the two triangles containing the first and the last edge of the connector (respectively) and the vertex adjacent to that edge but outside of the connector; the set $V_C$ contains the remaining $2d-2$ vertices on each of the connectors; the set $V_B$ contains all vertices that are incident to any of the nontriangular faces and are not in $V_C\cup V_D$; lastly, $V_A$ contains all remaining vertices.
  Note that $|V_D| = 18k$, $|V_C| = 6k(d-1)$, $|V_B| = 2k(m-6)$, and $|V_A| = \nu(g,k)-2k(m+3)$.

  Suppose that $\Pi$ is a 2-cell embedding of $G$ whose genus is (at most) $g+r$. We will compare which facial triangles under $\Pi_0$ are no longer facial under $\Pi$. We say that a $\Pi_0$-facial triangle is \emph{unstable} if it is not $\Pi$-facial; otherwise it is \emph{stable}. We say that a vertex $v\in V(G)$ is \emph{unstable} if it is contained in an unstable triangle. In the first part of the proof, we will show that there are not too many unstable vertices in $V_A\cup V_B\cup V_D$. By ``guessing'' which vertices in $V_A\cup V_B\cup V_D$ are unstable, we get a bound on the number of possible embeddings $\Pi$ having the guessed unstable set. On the other hand, the number of unstable vertices in $V_C$ can be much larger. For the number of rotation systems on the connectors we will need more specific homotopy arguments.

  Let us first consider the set $U$ of all unstable vertices in $V_A\cup V_B$. Let $U'$ be the largest set of unstable triangles, each of which containing at least one vertex in $V_A$ and such that any two triangles in $U'$ are at distance at least 2 in $G$. It is easy to see that the only possible cofacial connector of any two triangles in $U'$ is one of the long faces and such a long face is unique for each triangle in $U'$. (Note that the possible exceptions to this property would be the triangles with their vertices in $V_C\cup V_D$. This was actually the reason to define $V_D$.) Consequently, the set $U'$ is sparse, and Lemma \ref{lem:non cofacial faces} shows that
  $$|U'|\le g+r.$$
  By maximality of $U'$, each vertex in $U$ is at distance at most 2 from $U'$. Since $G$ has maximum degree 7 and vertices of degree 7 are at distance at least 4 from each other, it is easy to see that
  $$|U|\le 30|U'|.$$
  Among the 30 or fewer vertices at distance at most 2 from a triangle in $U'$, there is at most one of degree 7 (by property (d)), and thus they give fewer than $6!(5!)^{29}$ different choices for rotations around these vertices. This ``overcount'' also counts possibilities that stable triangles can have two orientations as the faces in $\Pi$.
  This yields that there are less than
  $$
     \binom{\nu(g,k)}{|U'|}\cdot \left(6 \cdot 6! \cdot (5!)^{29}\right)^{|U'|}
  $$
  ways to select rotations around vertices in $V_A\cup V_B$. Similarly, vertices in $V_D$ have at most $(5!)^{18k}$ ways to choose local rotations around them.

\begin{figure}
\begin{center}
\includegraphics[width=102mm]{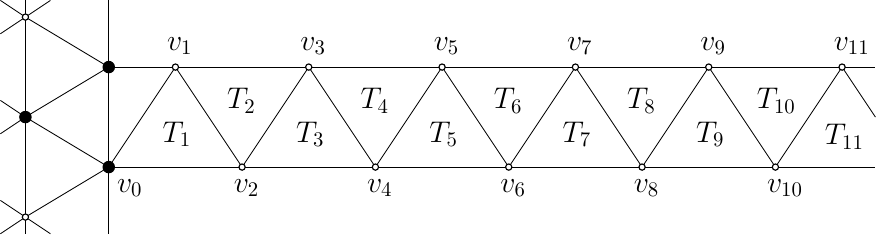}
\end{center}
\caption{Triangles on a connector (drawn horizontally). The three black vertices belong to $V_D$.}
\label{fig:3}
\end{figure}

  To count the number of ways to change the local rotations around vertices in $V_C$ and still have genus at most $g+r$, we first note that the above bounds and properties (e) and (f) show that the genus of the subgraph at distance $f(g,k)$ from the long faces of $G$ is at least $g$. Thus we can afford to increase the genus at most by $r$ when changing the rotation around the vertices on the connectors.

  Let us now consider one of the connectors. We will use the notation shown on Fig.~\ref{fig:3}, where the connector is drawn horizontally. Starting on the left with triangle $T_1$, we find the first unstable triangle $T_{i_1}$. Suppose that there is another unstable triangle $T_j$ that is homotopic to $T_{i_1}$ under the embedding $\Pi$. (Here the homotopy is free homotopy among some orientation of two triangles.) We let $T_{i_2}$ be the homotopic triangle with largest possible $i_2$, possibly $i_2=i_1$. Since the two homotopic noncontractible triangles bound a cylinder, any triangle $T_j$ with $i_1\le j\le i_2$ is either contractible or homotopic to $T_{i_1}$. Then we look further and find the smallest $i_3\ge i_2+3$ such that $T_{i_3}$ is unstable. That unstable triangle is disjoint from $T_{i_1}$ and $T_{i_2}$. If $i_3$ exists, then we define $i_4$ and $T_{i_4}$ in the same way as we did to obtain $i_2$ after having fixed $T_{i_1}$. Proceeding, we obtain pairs $(i_1,i_2), (i_3,i_4), \dots, (i_{2t-1},i_{2t})$. The triangles $T_{i_1},T_{i_3},\dots, T_{i_{2t-1}}$ are pairwise disjoint and nonhomotopic, and they increase the genus at most by $r$. Thus, by Lemma \ref{lem:disjoint nonhomotopic}, $t\le 3r-2$.

\begin{figure}
\begin{center}
\includegraphics[width=144mm]{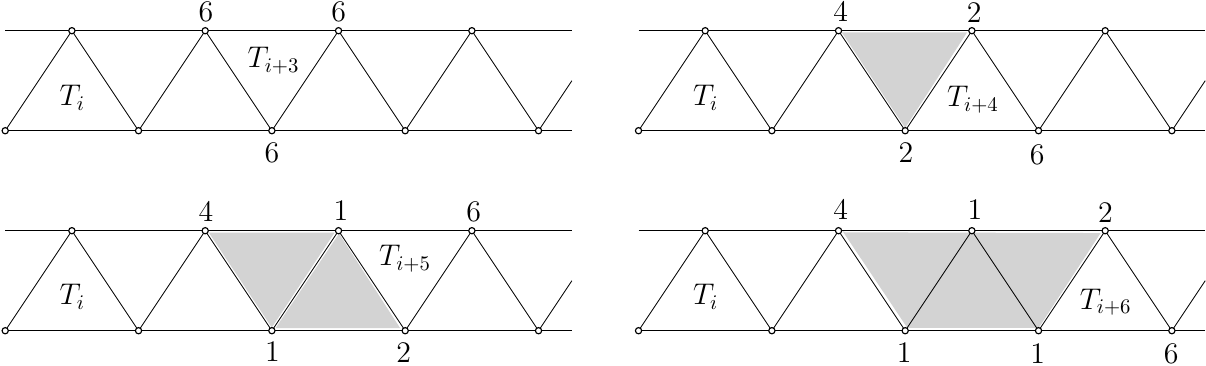}
\end{center}
\caption{Degrees of freedom when the homotopy of unstable triangles is changed. The shaded triangles are facial (with the shown or with the opposite orientation on the surface).}
\label{fig:4}
\end{figure}

  The above paragraph shows that we can determine the number of possibilities of different local rotations on one of the connectors by first finding $i_1,i_2,\dots,i_{2t-1},i_{2t}$ (less than $(2d)^{6r-4}$ possibilities), then estimating that each edge between $T_{i_{2j-1}}$ and $T_{i_{2j}}$ has two choices for joining the two paths on the connector, which amounts all together to at most $2^{2d}$ possibilities, and finally, when changing homotopy, considering the possible ways of local rotations at vertices between $T_{i_{2j}}$ and $T_{i_{2j+1}}$. The latter possibility gives at most $6^3$ ways, for local rotations, including for the three vertices on $T_{i_{2j+1}}$. Figure \ref{fig:4} shows different possibilities for the latter count, where $6^3$ is the worst. The $6^3$ ways can be replaced by $3^3$ to compensate with the overcount of $2^{2d}$ stated above.

  Without trying to optimize and use the same bounds on other connectors, we take it from the proof of Claim \ref{cl:lower bound} that at most $q+r \le 3k$ connectors have unstable triangles. There are at most $2^{3k}$ possible subsets of connectors with unstable triangles. This yields that the number of embeddings is at most
  $$
     2^{3k}\Bigl(2^{2d} \bigl(2d\cdot 3^6\bigr)^{6r-4}\Bigr)^{q+r} .
  $$

  All over, we obtain the upper bound of
  \begin{eqnarray*}
     \binom{\nu}{|U'|} \cdot \left(6 \cdot 6! \cdot (5!)^{29}\right)^{|U'|} && \cdot (5!)^{18k} \cdot
     2^{3k} \Bigl(2^{2d} \bigl(2d\cdot 3^3\bigr)^{6r-4}\Bigl)^{q+r} \\
     &\le& \left(36 (5!)^{30} \nu\right)^{g+2k} \cdot \left( 2 (5!)^6 (54d)^{12k-4} \right)^{3k} \cdot 2^{2d(q+r)}  \\
     &<& \left((5!)^{31} \nu\right)^{g+2k} \cdot \left(108d\right)^{36k^2}
     \cdot 2^{2d(q+r)} ,
  \end{eqnarray*}
  where $\nu=\nu(g,k)$.
  \end{proof}

Finally, suppose that $1 \le r < 2k$, where $r$ is odd. Then, using Claim \ref{cl:lower bound} for $r-1$ and $r+1$ (and noting that $r-1$ and $r+1$ are even), we see that
\begin{equation}\label{eq:1}
   a_{g+r-1}(G_{g,k})\cdot a_{g+r+1}(G_{g,k}) > 2^{6dr}.
\end{equation}
By using Claim \ref{cl:upper bound} for $r$, we have
\begin{equation}\label{eq:2}
   a_{g+r}(G_{g,k}) \le \left((5!)^{31} \nu\right)^{g+2k} \cdot (108d)^{36k^2} 2^{(3r-1)d} .
\end{equation}
Comparing (\ref{eq:1}), the square of (\ref{eq:2}) and recalling (\ref{eq:d}), the conclusion of the theorem follows easily.
\end{proof}

\section{Many consecutive log-convex terms}

One obvious question is whether having only every second term being strongly log-convex is necessary. Our examples can be generalized to obtain the following result.

\begin{theorem}\label{thm:second main result}
  For every $g\ge0$ and $k\ge1$, there exists a $4$-connected graph of genus $g$, whose genus distribution sequence has $k$ consecutive terms that are strongly log-convex.
\end{theorem}

\begin{proof}
  The proof is essentially the same as that of Theorem \ref{thm:main} and we do not give all details here. However, we provide the basic idea.

  The graphs for this theorem are essentially the same as those in the proof of Theorem \ref{thm:main}. The graphs $G_{g,k}$ have $k$ connecting parts, in each of which three connectors join one stacked antiprism with another one. We replace the $j$th triple of connectors with $k+j$ connectors ($1\le j\le k$). We also make their lengths slightly different, the lengths increasing with $j$ (this will make the proof work). For some very large $d$, we let the length of all these $q=k+j$ connectors be $c_qd$, where
  $$
      c_q = \frac{q-1}{q} + \frac{q}{4k^2}, \qquad k\le q\le 2k.
  $$
  Note that $c_k = 1-\tfrac{3}{4k} < c_{k+1} < \cdots < c_{2k} = 1$ and that the function $r(q) = q.c_q$ is strictly convex:
  \begin{equation}\label{eq:qc_q is convex}
    2qc_q < (q-1)c_{q-1} + (q+1)c_{q+1}.
  \end{equation}
  Details are left to the reader.

  The genus coefficients $a_g,a_{g+1},\dots, a_{g+k-1}$ grow ``roughly" with the factor of $2^{(1+o_d(1))d}$, i.e.
  $$
      a_{g+t} = 2^{(1+o_d(1))dt}
  $$
  for $1 \le t < k$. This growth is realized by joining $t+1$ large faces of the last connectors where we have $2k+1$ connectors of height $c_{2k}d=d$.
  However, $a_{g+k},a_{g+k+1},\dots, a_{g+2k}$ grow exponentially faster. We claim that
  \begin{equation}\label{eq:consecutive log-cvx}
      a_{g+k+t} = 2^{c_{k+t+1}d(k+t+1)(1+o_d(1))}
  \end{equation}
  for $0 \le t \le k$. To see that $a_{g+k+t} \ge 2^{c_{k+t+1}d(k+t+1)}$, we use a similar proof as for Claim \ref{cl:lower bound}. With $k+t+1$ connectors of height $c_{k+t+1}d$ we increase genus $g$ to $g+k+t$ to obtain one face where the sides of these $k+t+1$ connectors appear as in the mentioned proof, yielding $2^{c_{k+t+1}d(k+t+1)}$ different embeddings of the same genus.

  To prove that the right hand side of (\ref{eq:consecutive log-cvx}) is an upper bound, we have to prove that using the increase of the genus by $k+t$ as in the lower bound construction above is best possible. To show this we have to argue that merging $k+t+1$ big faces with the highest connectors of height $c_{k+t+1}d$ gives much smaller numbers, and that using the gain where we merge $k+s$ faces of height $c_{k+s}d$ and adding the rest of the genus by merging $t-s+1$ faces with the connectors of height $c_{2k}d$ also gives lower numbers. This comparison reduces to the following two inequalities:
  \begin{eqnarray*}
    c_{k+t+1}(k+t+1) &>& c_{2k}(k+t)\quad \mathrm{and} \\
    c_{k+t+1}(k+t+1) &>& c_{k+s+1}(k+s) + c_{2k}(t-s), \quad 0\le s < t.
  \end{eqnarray*}
  Using the fact that $c_{2k}=1$, they are both easy to verify and details are left to the reader.

  Finally, (\ref{eq:consecutive log-cvx}) implies that for every $0\le t <k$ the following holds
  \begin{eqnarray*}
    2\log_2 a_{g+k+t} &\approx& 2c_{k+t+1}d(k+t+1) \\
     &\ll& c_{k+t}d(k+t) + c_{k+t+2}d(k+t+2) \\
     &\approx& \log_2 a_{g+k+t-1}(1+o_d(1)) + \log_2 a_{g+k+t+1}(1+o_d(1)),
  \end{eqnarray*}
  where $\approx$ means equality up to a factor $1+o(1)$, and $a\ll b$ means that $a/b < 1-\varepsilon$ for some $\varepsilon>0$. This harsh inequality follows from (\ref{eq:qc_q is convex}). This completes the proof.
\end{proof}

\section{In conclusion}

Let us end with a more positive viewpoint.
Our counterexamples to the LCGD Conjecture are still unimodal. As far as the author can tell, the unimodal weakening of the LCGD Conjecture could (and actually should) still be true.

\begin{conjecture}
  The genus distribution of every graph $G$ is unimodal.
\end{conjecture}

The following also looks plausible.

\begin{conjecture}
  Every graph that triangulates some surface has log-concave genus distribution.
\end{conjecture}

\bibliographystyle{plain} 

\bibliography{Genus.bib}

\end{document}